\documentclass[letterpaper,11pt,twoside,reqno]{amsart}
\usepackage{times}
\usepackage{amsmath}
\usepackage{amssymb}
\usepackage{amsfonts}
\usepackage{amscd}
\usepackage{amsthm}
\usepackage{amsxtra}
\usepackage{stmaryrd}
\usepackage[all]{xy}  \CompileMatrices
\usepackage{mathrsfs}
\usepackage{nicefrac}
\usepackage{graphicx}
\usepackage{color}
\usepackage[raggedright,IT,hang]{subfigure}
\usepackage{wrapfig}
\usepackage{enumerate}
\usepackage{textcomp}

\renewcommand{\paragraph}[1]{\par\vspace{1ex}\noindent #1}
\newcommand{\scalProd}[2]{\langle{#1},{#2}\rangle}

\theoremstyle{plain}
\newtheorem{theorem}{Theorem}

\newtheorem*{proposition*}{Proposition}

\newtheorem*{corollary*}{Corollary}
\newtheorem{lemma}[theorem]{Lemma}

\newtheorem*{theorem*}{Theorem}
\newtheorem*{lemma*}{Lemma}
\newtheorem*{conjecture*}{Conjecture}

\newtheorem*{question*}{Question}

\theoremstyle{definition}

\newtheorem*{exercise*}{Exercise}

\theoremstyle{remark}

\newtheorem*{remark*}{Remark}
\newtheorem{remsTh}[theorem]{Remarks}

\newcommand{\subclass}[1]{}

\newcommand{\enumTi}[1]{\renewcommand{\theenumi}{#1}}

\newcommand{\alphenumi}{\enumTi{\alph{enumi}}}

\newcommand{\romenumi}{\enumTi{\roman{enumi}}}

\DeclareMathOperator{\id}{id}

\newcommand{\sabs}[1]{{\lvert{#1}\rvert}}

\newcommand{\RR}{\mathbb{R}}

\newlength{\algotabbingwidth}
\newcommand{\setDef}[2]{\{{#1}\,:\,{#2}\}}
\newcommand{\ints}[1]{[{#1}]}
\newcommand{\symGr}[1]{\mathfrak{S}({#1})}
\newcommand{\altGr}[1]{\mathfrak{A}({#1})}

\newcommand{\charVec}[1]{\chi({#1})}

\DeclareMathOperator{\isoOp}{iso}
\newcommand{\isoGr}[1]{\isoOp({#1})}
\DeclareMathOperator{\POp}{P}
\DeclareMathOperator{\cardOp}{card}
\newcommand{\PCard}[1]{\POp_{\cardOp}({#1})}

\newcommand{\upop}{\varkappa}

\begin{document}

\title[Tight Lower Bounds on the Sizes of Symmetric Extensions of Permutahedra]{Tight Lower Bounds on the Sizes of Symmetric Extensions of Permutahedra and Similar Results}
\author{Kanstantsin Pashkovich}

\thanks{Supported by  the Progetto di Eccellenza 2008-2009 of the Fondazione Cassa Risparmio di Padova e Rovigo.}
\address{Dipartimento di Matematica, Universit\`a degli Studi di Padova, Via Trieste 63, 35121 Padova, Italy}
\email{pashkovich@math.unipd.it}

\begin{abstract}
It is well known that the permutahedron $\Pi_n$ has $2^n-2$ facets. The Birkhoff polytope provides a symmetric extended formulation of $\Pi_n$ of size $\Theta(n^2)$. Recently, Goemans described a non-symmetric extended formulation of $\Pi_n$ of size $\Theta(n\log n)$. In this paper, we prove that $\Omega(n^2)$ is a lower bound for the size of symmetric extended formulations of $\Pi_n$. Moreover, we prove that the cardinality indicating polytope has the same tight lower bounds for the sizes of symmetric and non-symmetric extended formulations as the permutahedron.
\end{abstract}
\maketitle

\section{Introduction}

Extended formulations of polyhedra have gained importance in the recent past, because this concept allows to represent a polyhedron by a higher-dimensional one with a simpler description. Thus, an optimization problem over an initial polytope can be easily transformed into an optimization problem over its extension. In some cases where the initial polytope has a complicated description in the initial space or even no such description is known, the reformulations via extended formulations appear to be helpful.

To illustrate the power of extended formulations let us take a look at the permutahedron~$\Pi_n\subseteq \RR^{n}$, which is the convex hull of all points obtained from the point~$(1,2,\ldots, n)\in \RR^n$ by coordinate permutations. The minimal description of~$\Pi_n$ in the space~$\RR^n$ looks as follows~\cite{CCZ09}:
\begin{align*}
\sum_{v=1}^{n}x_v&=\frac{n(n+1)}{2}&&\\
\sum_{v\in S}x_v&\ge\frac{|S|(|S|+1)}{2}&&\text{ for all }\varnothing\neq S\subset\ints{n}\,.
\end{align*}
Thus the permutahedron~$\Pi_n$ has~$n!$ vertices and~$2^n-2$ facets. At the same time it is easy to derive an extended formulation of size~$\Theta(n^2)$  from the Birkhoff polytope~\cite{CCZ09}:
\begin{equation}\label{eq:sym_ext_perm}
\begin{aligned}
\sum_{i=1}^n i z_{i,v}&=x_v&&\text{ for all }v\in\ints{n}\\
\sum_{v=1}^n z_{i,v}&=1&&\text{ for all }i\in\ints{n}\\
\sum_{i=1}^n z_{i,v}&=1&&\text{ for all }v\in\ints{n}\\
z_{i,v}\ge 0&&&\text{ for all }i,v\in\ints{n}\,.
\end{aligned}
\end{equation}
The projection of the polyhedron described by~\eqref{eq:sym_ext_perm} to the~$x$-variables gives the permutahedron~$\Pi_n$.  Clearly, every coordinate permutation of~$\RR^n$ maps~$\Pi_n$ to itself. The extended formulation~\eqref{eq:sym_ext_perm} respects this symmetry in the sense that every such permutation of the~$x$-variables can be extended by some permutation of the~$z$-variables such that these two permutations leave~\eqref{eq:sym_ext_perm} invariant (up to reordering of the constraints).

Also there exists a non-symmetric extended formulation of the permutahedron of size~$\Theta(n\log(n))$~\cite{Goe09}. This is the best one can achieve~\cite{Goe09} due to the fact that every face of~$\Pi_n$ (including the~$n!$ vertices) is a projection of some face of the extension. And since the number of faces of a polyhedron is bounded from above by~$2$ to the number of its facets, we can conclude that every extension of the permutahedron has at least~$\log_2(n!)=\Theta(n\log(n))$ facets. 

Another illustrative example is the cardinality indicating polytope~$\PCard{n}$, which is defined as the convex hull of vectors $(x,z)$ for which the following equations hold
\begin{equation*} 
 \sum_{i=1}^{n} x_i=\sum_{j=0}^{n} j z_{j+1} \quad\text{and}\quad\sum_{j=0}^{n}z_{j+1}=1
\end{equation*}
and components satisfy $x\in\{0,1\}^n$ and $z\in\{0,1\}^{n+1}$.
The cardinality indicating polytope~$\PCard{n}$ has a minimal description in the initial space $\RR^{2n+1}$ given by the following linear constraints~\cite{KLW05}:
\begin{equation}\label{eq:initial_card}
\begin{aligned}
&\sum_{i\in S}x_i \le \sum_{j=0}^{|S|} j z_{j+1} + |S|\sum_{j=|S|+1}^{n}z_{j+1}  &&\text{ for all }\varnothing\neq S\subset\ints{n}\\
&\sum_{i=1}^{n} x_i=\sum_{j=0}^{n} j z_{j+1}\\ &\sum_{j=0}^{n}z_{j+1}=1\\
&x_i\ge 0,\qquad z_j\ge 0 &&\text{ for all }i\in\ints{n}\text{ and }j\in\ints{n+1}.
\end{aligned}
\end{equation}
This shows that the cardinality indicating polytope has exponentially many facets.

 One can construct a symmetric extended formulation of the cardinality indicating polytope of size~$\Theta(n^2)$:
\begin{equation}\label{eq:sym_ext_card}
\begin{aligned}
&\sum_{i=1}^n y_{i,j}=(j-1)z_j&&\text{ for all } j\in\ints{n+1}\\
&\sum_{j=1}^{n+1}y_{i,j}=x_i&&\text{ for all }i\in\ints{n}\\
&\sum_{j=0}^{n}z_{j+1}=1\\
&0\le y_{i,j}\le z_j &&\text{ for all }i\in\ints{n}\text{ and } j\in\ints{n+1}\,,
\end{aligned}
\end{equation}
which represents the Balas extension~\cite{Bal85} for the faces of the cardinality indicating polytope induced by equations $z_j=1$, $j\in\ints{n+1}$.
There is also  a non-symmetric extended formulation of size $\Theta(n\log(n))$~\cite{KP10}, which is constructed in the similar way to the non-symmetric extended formulation of the permutahedron.

Moreover the cardinality indicating polytope~$\PCard{n}$ has a set of~$n!$ non-trivial faces, which shows that every extended formulation of the cardinality indicating polytope involves~$\Omega(n\log(n))$ inequalities. The mentioned non-trivial faces of the cardinality indicating polytope~$\PCard{n}$ are indexed by all possible permutations~$\mu$ of~$n$ elements and each of those faces is the intersection of the following~$n-1$ facets:
\begin{align*}
\sum_{\mu(v)\in \ints{q}}x_v-\sum_{k=0}^{q} k z_{k+1}-q \sum_{k=q+1}^{n}  z_{k+1} = 0\text{ for all } 1 \le q \le n-1.
\end{align*}
One can easily check that two such faces of the cardinality indicating polytope are different whenever they correspond to two different permutations. Indeed, having two different permutations~$\mu'$ and~$\mu''$ we can find~$q$ with~$1 \le q \le n-1$, such that~$\mu'^{-1}(\ints{q})$ is not equal to~$\mu''^{-1}(\ints{q})$. Then the vertex $(x,z)$ of the cardinality indicating polytope~$\PCard{n}$ defined by:
\begin{equation*}
 \begin{aligned}
&x_i=1&&\text{ if } i\in\mu'^{-1}(\ints{q})\\
&x_i=0&&\text{ otherwise }\\
&z_j=1&&\text{ if } j=q+1\\
&z_j=0&&\text{ otherwise }\\
\end{aligned}
\end{equation*}
belongs to the face indexed by the permutation~$\mu'$, but does not belong to the face indexed by~$\mu''$.

 As we show in this paper the size of the extended formulations~\eqref{eq:sym_ext_perm} and~\eqref{eq:sym_ext_card} are asymptotically optimal for symmetric formulations of the permutahedron and the cardinality indicating polytope. Thus there exists a gap in the size between symmetric and non-symmetric extended formulations of~$\Pi_n$ and of~$\PCard{n}$. This situation appears in some other cases as well, e.g. the cardinality constrained matching polytopes and the cardinality constrained cycle polytopes~\cite{KPT09}. But even if the gaps observed in those cases are more substantial, the permutahedron and the cardinality indicating polytope  are interesting because of the possibility to determine tight asymptotical lower bounds~$\Omega(n^2)$ and ~$\Omega(n\log(n))$ on the sizes of symmetric and non-symmetric extended formulations, respectively.

The paper is organized as follows. Section~\ref{sec:definitions} contains definitions of extensions, the crucial notion of section and some auxiliary results. In section~\ref{sec:symextension} we exploit the techniques from~\cite{KPT09},\cite{Yan91} and some new approaches to investigate the structure of symmetric extensions of certain size. In sections~\ref{sec:symextension_perm} and~\ref{sec:symextension_card} we apply our knowledge about the structure to get lower bounds on the size of symmetric extensions of the  permutahedron and the cardinality indicating polytope.
\paragraph{Acknowledgements.} I thank Volker Kaibel for valuable comments which led to simplifications in the proofs and for his useful recommendations on wording. I also thank the referee for his careful reading and constructive comments which helped to increase the readability of the presented paper.

\section{Extensions, Sections and Symmetry}
\label{sec:definitions}

Here we list some known definitions and results, which will be used later. For a broader discussion of symmetry in extended formulations of polyhedra we refer the reader to~\cite{KPT09}.

A polytope~$Q\subseteq\RR^d$ together with a linear map~$p:\RR^d\rightarrow\RR^m$ is called an \emph{extension} of a polytope~$P\subseteq\RR^m$ if the equality~$p(Q)=P$ holds, the \emph{size} of the extension~$Q$, $p$ is the number of facets of the polytope~$Q$. Moreover, if~$Q$ is the intersection of an affine subspace of~$\RR^d$ and the nonnegative orthant~$\RR_+^d$ then~$Q$ is called a \emph{subspace extension}. A (finite) system of linear equations and inequalities whose solutions are the points in~$Q$ is an \emph{extended formulation} for~$P$, the \emph{size} of the extended formulation is the number of inequalities in the system.

Throughout the paper we heavily deal with \emph{sections} ~$s:X\rightarrow Q$, which are maps  that assign to every vertex~$x\in X$ of~$P$  some point~$s(x)\in Q\cap p^{-1}(x)$. Such a section induces a bijection between the vertex set~$X$ of the polytope~$P$ and points~$s(X)$ in the polytope~$Q$, whose inverse map is given by~$p$.

Let~$G$ be a group with the group operation $\circ:G\times G\rightarrow G$ acting on the set~$X$ of vertices of~$P$. In other words, every group element $\pi \in G$ defines a map of $X$ on itself (for every $x \in X$ the image of $x$ under this map is denoted by $\pi.x$), satisfying:
\begin{enumerate}
  \item $(\pi \circ \sigma).x$ equals $\pi.(\sigma.x)$ for every $x\in X$ and $\pi$, $\sigma\in G$,
  \item the identity element of $G$ maps every vertex in $X$ on itself.
\end{enumerate}

 In this setting, an extension is called \emph{symmetric} with respect to the action of~$G$ on~$X$, if for every~$\pi\in G$ there is an affine isomorphism~$\upop_{\pi}:\RR^d\rightarrow\RR^d$
 with~$\upop_{\pi}.Q = Q$
and
\begin{equation}\label{eq:pGmap}
	p(\upop_{\pi}.y) = \pi.p(y) \quad\text{for all~$y\in p^{-1}(X)$.}
\end{equation}
The extension is called \emph{coordinate-symmetric} if the corresponding affine maps~$\upop_\pi$ can be chosen to be coordinate permutations, i.e. for all $\pi\in G$ we have $\upop_\pi\in\symGr{d}$, where $\symGr{d}$ is the group of all permutations of elements in $\ints{d}$.

We define an extended formulation~$A^=y=b^=$,~$A^{\le}y\le b^{\le}$ describing the polyhedron~$Q$, i.e.
\begin{equation*}
	Q=\setDef{y\in\RR^d}{A^=y=b^=, A^{\le}y\le b^{\le}}
\end{equation*}
 to be \emph{symmetric} (with respect to the action of~$G$ on the set~$X$ of vertices of~$P$) if for every~$\pi\in G$ there is an affine isomorphism~$\upop_{\pi}:\RR^d\rightarrow\RR^d$, $\upop_\pi(y)=C_\pi.y+c_\pi$, satisfying~\eqref{eq:pGmap} and such that the matrices $(A^=.C_\pi, b^=-A^=.c_\pi)$ and~$(A^\le.C_\pi, b^\le-A^\le.c_\pi)$ are equal to the matrices~$(A^=,b^=)$ and~$(A^{\le}, b^{\le})$  up to possible row reorderings. A \emph{coordinate-symmetric} extended formulation is the symmetric extended formulation with $\upop_\pi\in\symGr{d}$ for all $\pi\in G$.

Clearly, in the case of a symmetric extended formulation the maps~$\upop_{\pi}$ satisfy~$\upop_{\pi}.Q= Q$, which implies the following lemma.

\begin{lemma}\label{lem:symExtFormToSymExt}
	Every symmetric extended formulation (coordinate-symmetric extended formulation) describes a symmetric extension (coordinate-symmetric extension).
\end{lemma}

On the other hand,  every symmetric extension can be transformed into a subspace coordinate-symmetric extension of a smaller or the same size (see \cite{KPT09}).
\begin{lemma}\label{thm:transform_eq}
	Every symmetric extension induces a coordinate-symmetric subspace extension of a smaller or the same size. 
\end{lemma}
Thus, a lower bound on the number of variables in coordinate-symmetric subspace extensions of the given polytope~$P$ provides a lower bound on the size of symmetric extensions for~$P$ (note, that the size of a subspace extension is at most the dimension of the ambient space, since every subspace extension can be defined by a set of equations and a set of the non-negativity constraints). 

To prove lower bounds for subspace extensions one may use the following fact~\cite{KPT09}: if $Q$ is a subspace extension of the polytope $P$ with a section $s:X\rightarrow Q$ and $\scalProd{a}{x}\le b$ a valid inequality for $P$ then the system:
\begin{align}
	\sum_{x\in X}s_j(x)\lambda_x&\ge 0\quad\text{for all  }j\in\ints{d} \label{eqn_section-slack-covectors-extension}\\
	\sum_{x\in X}(b-\scalProd{a}{x})\lambda_x&<0\label{eqn_section-slack-covectors-polytope}
\end{align}
does not have a solution $\lambda\in\RR^X$.

A section~$s:X\rightarrow Q$ is \emph{coordinate-symmetric} (with respect to the action of~$G$ on~$X$) if for every~$\pi\in G$ there is a permutation~$\upop_{\pi}\in\symGr{d}$ with~$s(\pi.x)=\upop_{\pi}.s(x)$ for all~$x\in X$. It can be shown that every coordinate-symmetric extension admits a coordinate-symmetric section~\cite{KPT09}. 

For a coordinate-symmetric section~$s:X\rightarrow Q\subseteq\RR^d$ we can define an action of~$G$ on the set~$\mathcal{S}=\{s_1,\dots,s_d\}$ of the component functions of the section~$s:X\rightarrow Q$ via~$\pi.s_j=s_{\upop_{\pi^{-1}}^{-1}(j)}\in\mathcal{S}$
for each~$j\in\ints{d}$. Before showing that this definition yields a group action observe that for each~$j\in\ints{d}$ the following holds
\begin{equation}\label{eq:actionCompSect}
	(\pi.s_j)(x)=s_{\upop_{\pi^{-1}}^{-1}(j)}(x)=(\upop_{\pi^{-1}}.s(x))_j=s_j(\pi^{-1}.x)\quad\text{for all }x\in X\,.
\end{equation}
Now from~\eqref{eq:actionCompSect}, we deduce that~$(\pi\circ\sigma).s_j=\pi.(\sigma.s_j)$ for all~$\pi,\sigma\in G$ and~$\id.s_j=s_j$ for the identity element~$\id$ in~$G$, which shows that the action of the group $G$ on $\mathcal{S}$ is defined correctly.

The \emph{isotropy group} of~$s_j\in\mathcal{S}$ under this action is defined as the following subgroup of $G$
\begin{equation*}
	\isoGr{s_j}=\setDef{\pi\in G}{\pi.s_j=s_j}\,.
\end{equation*}
Thus an element~$\pi$ of $G$ is in the isotropy group~$\isoGr{s_j}$ if and only if ~$s_j(x)=s_j(\pi^{-1}.x)$ holds  for all~$x\in X$ (or equivalently, ~$s_j(\pi.x)=s_j(x)$  for all~$x\in X$). The \emph{orbit} corresponding to the component function~$s_i\in\mathcal{S}$ under the action of the isotropy group $\isoGr{s_j}$ is the following subset of $\mathcal{S}$
\begin{equation*}
    \setDef{\pi.s_i}{\pi\in \isoGr{s_j}}\,.
\end{equation*}

In general, it is impossible to determine the isotropy groups~$\isoGr{s_j}$ without more knowledge on the section~$s$. However, for each isotropy group~$\isoGr{s_j}$ it is possible to bound its index
\begin{equation*}
    (G:\isoGr{s_j})={\sabs{G}}/{\sabs{\isoGr{s_j}}}\,,
\end{equation*}
since the index is equal to the number of orbits under the action of $\isoGr{s_j}$ on $\mathcal{S}$, and thus is bounded from above  by the total number of variables, i.e.
\begin{equation*}
		(G:\isoGr{s_j})\le d\,.
\end{equation*}
To identify suitable subgroups of the isometry group~$\isoGr{s_j}$ one may use the above bound on the index together with the following result on subgroups of the symmetric group~$\symGr{n}$~\cite{Yan91}.

\begin{lemma}\label{lem:Yannakakis-claim}
	 For each subgroup~$U$ of~$\symGr{n}$ with~$(\symGr{n}:U)\le\binom{n}{k}$ for ~$k<\frac{n}{4}$, there is some ~$W\subseteq\ints{n}$ with~$|W|\le k$ such that 
	\begin{equation*}
		\setDef{\pi\in
		  \altGr{n}}{\pi(v)=v\text{ for all }v\in W}\subseteq U
	\end{equation*}
	holds.	
\end{lemma}

\section{Symmetric Subspace Extensions of Quadratic Size }
\label{sec:symextension}

The main result of this section is Theorem~\ref{thm:partition}, which describes the action of the group~$\altGr{n}$ on the components~$s_j$. Here~$\altGr{n}$ denotes the alternating group, i.e. the group consisting of all even permutations on the set~$\ints{n}$.  

Consider a polytope~$P\subseteq\RR^{n+m}$  for~$n\ge 9$, where the group~$\symGr{n}$ acts on the set~$X$ of vertices of~$P$ by permuting the first $n$~coordinates. Let a polytope~$Q\subseteq\RR^d$ with~$2d<n(n-1)$ be a  subspace extension of the polytope~$P$ with a coordinate-symmetric section. The mentioned coordinate-symmetric section~$s:X\rightarrow Q$ is defined with respect to the action of~$G=\symGr{n}$  on the vertex set~$X$. 

\begin{lemma}\label{lem:one_node}
For each~$j\in\ints{d}$ there is~$v_j\in \ints{n}$ such that
\begin{equation*}
 	\setDef{\pi\in\altGr{n}}{\pi(v_j)=v_j}\subseteq\isoGr{s_j}
\end{equation*}
This element~$v_j$ is uniquely determined unless~$\altGr{n} \subseteq \isoGr{s_j}$.
\end{lemma}

\begin{proof}
As we assumed~$d<\binom{n}{2}$ and thus Lemma~\ref{lem:Yannakakis-claim} implies that for all~$j\in\ints{d}$
\begin{equation*}
 	\setDef{\pi\in\altGr{n}}{\pi(v)=v\text{ for all }v\in V_j}\subseteq\isoGr{s_j}
\end{equation*}
 for some set~$V_j\subset \ints{n}$,~$|V_j|\le 2$. Thus, it has to be proven that~$V_j$ can be chosen to contain not more than one element, which we later denote by~$v_j$.

Let us assume that the set~$V_{j}$ consists of two elements~$\{u,w\}$ such that
\begin{multline*}
\setDef{\pi\in\altGr{n}}{\pi(v)=v\text{ for all }v\in \{u,w\}}\subseteq\isoGr{s_j}\qquad\text{and}\\
 	\setDef{\pi\in\altGr{n}}{\pi(u)=u}\not \subseteq\isoGr{s_j}\qquad\text{and}\\ \setDef{\pi\in\altGr{n}}{\pi(w)=w}\not \subseteq\isoGr{s_j}\,.
\end{multline*}

Due to~\eqref{eq:actionCompSect}, for every $\pi$, $\sigma\in\altGr{n}$ and for every $x\in X$ the following holds
\begin{equation*}
	(\sigma.s_j)(\pi.x)=s_j((\sigma^{-1}\circ\pi).x)\,.
\end{equation*}
Note that $s_j(\sigma^{-1}.x)=s_j(\sigma^{-1}\circ\pi.x)$ holds for all $x\in X$ if and only if $s_j(x)=s_j(\sigma^{-1}\circ\pi\circ\sigma.x)$ holds for all $x\in X$ since $\sigma$ defines an automorphism on $X$. Thus, we obtain the following
\begin{multline*}
 	\setDef{\pi\in\altGr{n}}{\pi(\sigma(v))=\sigma(v)\text{ for all }v\in \{u,w\}}\subseteq\isoGr{\sigma.s_j}\qquad\text{and}\\ 
	\setDef{\pi\in\altGr{n}}{\pi(\sigma(u))=\sigma(u)}\not \subseteq\isoGr{\sigma.s_j}\qquad\text{and}\\ \setDef{\pi\in\altGr{n}}{\pi(\sigma(w))=\sigma(w)}\not \subseteq\isoGr{\sigma.s_j}\,.
\end{multline*}

This shows that for every two element $u'$, $w'\in\ints{n}$ there is a coordinate function $s_{j'}\in\mathcal{S}$ such that
\begin{multline}\label{eq:fixingPairs}
 	\setDef{\pi\in\altGr{n}}{\pi(v)=v\text{ for all }v\in \{u',w'\}}\subseteq\isoGr{s_{j'}}\qquad\text{and}\\ 
	\setDef{\pi\in\altGr{n}}{\pi(u')=u'}\not \subseteq\isoGr{s_{j'}}\qquad\text{and}\\ \setDef{\pi\in\altGr{n}}{\pi(w')=w'}\not \subseteq\isoGr{s_{j'}}\,,
\end{multline}
since the alternating group $\altGr{n}$ is $2$-transitive, i.e. for every (possibly involving a common element) two pairs of elements in $\ints{n}$ there exists a  permutation in $\altGr{n}$ which maps the first pair on the second pair.

Since the number of different component functions in $\mathcal{S}$ is smaller then $\binom{n}{2}$ there exists a component function $s_{j'}$, which satisfies~\eqref{eq:fixingPairs} for two different pairs $(u',w')$ and $(u'',w'')$ of elements in~$\ints{n}$. Let us consider two different cases: these pairs have one element in common or these pairs are disjoint. 

Let us consider the first case and let $u'$ be equal $u''$. We obtain a contradiction to the statement
\begin{equation*}
    \setDef{\pi\in\altGr{n}}{\pi(u')=u'}\not \subseteq\isoGr{s_{j'}}\,,
\end{equation*}
since it is not hard to see that for every two distinct elements $w'$ and $w''$ the elements in two following subgroups of $\isoGr{s_{j'}}$
\begin{multline*}
    \setDef{\pi\in\altGr{n}}{\pi(v)=v\text{ for all }v\in \{u',w'\}}\quad\text{and}\\\setDef{\pi\in\altGr{n}}{\pi(v)=v\text{ for all }v\in \{u',w''\}}
\end{multline*}
together generate the group $\setDef{\pi\in\altGr{n}}{\pi(u')=u'}$.

In the second case, it is straightforward to show that for two disjoint pairs $(u',w')$ and $(u'', w'')$ the groups 
\begin{multline*}
    \setDef{\pi\in\altGr{n}}{\pi(v)=v\text{ for all }v\in \{u',w'\}}\quad\text{and}\\\setDef{\pi\in\altGr{n}}{\pi(v)=v\text{ for all }v\in \{u'',w''\}}
\end{multline*}
together generate the alternating group $\altGr{n}$, which contradicts~\eqref{eq:fixingPairs}. The same argumentation shows that $v_j$ is uniquely determined unless the isotropy group contains all even permutations.
\end{proof}

\begin{theorem}\label{thm:partition}
 	There exists a partition of the set~$\ints{d}$ into sets~$\mathcal{A}_1$,\ldots,$\mathcal{A}_\iota$ and~$\mathcal{B}$, such that each set~$\mathcal{A}_i$ consists of~$n$ elements~$a^i_1$,~$a^i_2$,\ldots,$a^i_n$ satisfying
	\begin{equation}\label{eq:partition}
	 	s_{a^i_t}(\pi.x)=s_{a^i_{\pi^{-1}(t)}}(x) \qquad\text{and}\qquad
	 	s_{b}(\pi.x)=s_{b}(x)
	\end{equation}
	for every ~$x\in X$, $b\in\mathcal{B}$, $\pi\in \altGr{n}$, $i\in\ints{\iota}$, $t\in\ints{n}$. 
\end{theorem} 
\begin{proof}
 Let us consider the orbit of a component function $s_j\in\mathcal{S}$ under the action of the alternating group $\altGr{n}$. There are two possible cases $\altGr{n}\subseteq\isoGr{s_j}$ and $\altGr{n}\not\subseteq\isoGr{s_j}$. In the first case, the component function $s_j$ is associated with the set $\mathcal{B}$.

In the second case, due to~\eqref{eq:actionCompSect} for every $\pi$, $\sigma\in\altGr{n}$ and for every $x\in X$ the following holds
\begin{equation*}
	(\sigma.s_j)(\pi.x)=s_j((\sigma^{-1}\circ\pi).x)\,.
\end{equation*}
Note that $s_j(\sigma^{-1}.x)=s_j(\sigma^{-1}\circ\pi.x)$ holds for all $x\in X$ if and only if $s_j(x)=s_j(\sigma^{-1}\circ\pi\circ\sigma.x)$ holds for all $x\in X$ since $\sigma$ defines an automorphism on $X$. 

Now let us use Lemma~\ref{lem:one_node} to show that for every $\sigma$ and $\phi$ in $\altGr{n}$ the component functions $\sigma.s_{j}$ and $\phi.s_j$ are identical whenever $\sigma(v_j)$ equals $\phi(v_j)$. Indeed, for every $x\in X$ the following holds
\begin{multline*}
 (\sigma.s_j)(x)=s_j((\sigma^{-1}).x)=s_j(\phi^{-1}\circ\sigma.(\sigma^{-1}.x))=\\s_j(\phi^{-1}\circ\sigma\circ\sigma^{-1}.x))=s_j(\phi^{-1}.x)=(\phi.s_j)(x)\,,
\end{multline*}
the second equality holds because the permutation $\phi^{-1}\circ\sigma$ is even and maps $v_j$ on itself, and thus lies in $\isoGr{s_j}$.

Moreover, for every $\sigma$ in $\altGr{n}$  the element $v_{j'}$(here  $s_{j'}=\sigma.s_j$) is uniquely defined and equals $\sigma(v_j)$. This follows in a straightforward manner from the fact that $(\sigma.s_j)(\pi.x)$ equals $(\sigma.s_j)(x)$ for every $x\in X$ if and only if $s_j((\sigma^{-1}\circ\pi\circ\sigma).x)$ equals $s_j(x)$ for every $x\in X$. 

Now, it is easy to see that it is enough to associate the coordinate functions $\pi.s_{j}$, $\pi\in\altGr{n}$ in the same orbit of $s_j$ to their elements $\pi(v_j)$ to finish the proof of the theorem.
\end{proof}

Let us establish the following theorem using Theorem~\ref{thm:partition} proved above.

\begin{theorem}\label{thm:partition_int}
 	If $X\subseteq{\{0,1\}}^n\times\RR^m$ there exists a partition of the set~$\ints{d}$ into sets~$\mathcal{A}_1$,\ldots,$\mathcal{A}_\iota$ and~$\mathcal{B}$, such that each set~$\mathcal{A}_i$ consists of~$n$ elements~$a^i_1$,~$a^i_2$,\ldots,$a^i_n$ satisfying
	\begin{align}
	 	&s_{a^i_v}(x)=s_{a^i_w}(x)&&\text{ if } x_v=x_w   \label{eq:partition_int}\\
	 	&s_{b}(x)=s_{b}(y)&&\text{ if } x=\pi.y\quad \text{ for some }\; \pi\in \altGr{n}
	\end{align}
	for every $x$, $y\in X$, $b\in\mathcal{B}$, $i\in\ints{\iota}$, $t\in\ints{n}$. 
\end{theorem} 
\begin{proof}
 Due to the definition of the section~$s$, for every component function $s_j\in\mathcal{S}$, permutation $\pi\in\altGr{n}$ and vertex $x\in X$ the value $s_j(x)$ equals $s_j(\pi.x)$, whenever $\pi$ lies  in the isotropy group $\isoGr{x}$. Moreover, for every two elements $v$ and $w$ with $x_v=x_w$ there exists a permutation $\pi$ in $\isoGr{x}\cap\altGr{n}$ with $\pi(v)=w$. Thus, the statement~\eqref{eq:partition_int} follows directly from Theorem~\ref{thm:partition}.
\end{proof}

\section{Permutahedron}
\label{sec:symextension_perm}

Now we would like to establish a lower bound on the
number of variables in symmetric subspace extensions of the permutahedron.
\begin{theorem}\label{symext:thm_lowerbound-perm}
	For every~$n\ge9$ there exists no symmetric extension of the permutahedron~$\Pi_{n}$ of size less than~$\frac{n(n-1)}{2}$ with respect to the group~$G=\symGr{n}$. 
\end{theorem}

Let us introduce the operator~$\Lambda(\zeta)$, which maps every permutation~$\zeta\in\symGr{n}$ to the vector ($\zeta^{-1}(1)$, $\zeta^{-1}(2)$, $\ldots$, $\zeta^{-1}(n)$). Thus, we have 
\begin{equation*}
	X=\setDef{\Lambda(\zeta)}{\zeta\in\symGr{n}}
\end{equation*} and
\begin{equation*}
	(\pi.\Lambda(\zeta))_v=\Lambda(\zeta)_{\pi^{-1}(v)}
\end{equation*}
for all~$\pi\in \symGr{n}$,~$\zeta\in\symGr{n}$,~$v\in\ints{n}$.

Let us assume the contrary, i.e. that there exists a symmetric subspace extension of the permutahedron of size less than~$\frac{n(n-1)}{2}$.  From Theorem~\ref{thm:partition} we have some understanding of how permutations in~$\altGr{n}$ act on the component functions of the section~$s$, which can be used to prove Theorem~\ref{symext:thm_lowerbound-perm}.

\begin{lemma}\label{symext:lem-exist-element}
 There exists~$w\in\ints{n-1}$ such that
\begin{equation}\label{symext:eqn_end}
	\text{if} \quad s_{a^i_w}(\Lambda(id_{n}))>0\quad \text{then} \quad \sum_{v>w} s_{a^i_v}(\Lambda(id_{n}))>0
\end{equation}
and
\begin{equation}\label{symext:eqn_start}
	\text{if} \quad s_{a^i_{w+1}}(\Lambda(id_{n}))>0 \quad\text{then}\quad \sum_{v\le w} s_{a^i_v}(\Lambda(id_{n}))>0
\end{equation}
hold for all $i\in\ints{\iota}$.
\end{lemma}
\begin{proof}

 Since each set~$\mathcal{A}_i$ consists of~$n$ components, we can conclude that~$\iota$ is less than~$\frac{n-1}{2}$     (recall~$d<\frac{n(n-1)}{2}$).

There can exist just one element~$u$ from~$\ints{n-1}$, which violates the statement (\ref{symext:eqn_end}) for a fixed index~$i\in\ints{\iota}$ (it can be only the maximal element from~$\ints{n-1}$ for which~$s_{a^i_u}(\Lambda(id_{n}))>0$, since the component functions take non-negative values). Analogously, for each~$i\in\ints{\iota}$ there can exist only one element~$u$ from~$\ints{n-1}$, which violates the statement (\ref{symext:eqn_start}). 

Thus, for at least one element~$w\in\ints{n-1}$ both (\ref{symext:eqn_end}) and (\ref{symext:eqn_start}) are satisfied for all~$i\in\ints{\iota}$.
\end{proof}

Let us introduce the following subgroups of~$\altGr{n}$ induced by elements~$w$ of the set $\ints{n-1}$
\begin{equation*}
 G_w=\setDef{\pi\in\altGr{n}}{\pi(\ints{w})=\ints{w}}\,,
\end{equation*}
i.e.~$G_w$ is the set of all even permutations of~$\ints{n}$ which map~$\ints{w}$ to itself.

To construct a contradiction to the assumption that there is an extension with the above properties we use~\eqref{eqn_section-slack-covectors-extension} and \eqref{eqn_section-slack-covectors-polytope}. Here, we choose $\lambda_x$, $x\in X$, $x=\Lambda(\zeta)$, as follows:
\begin{equation*}
	\lambda_x=\begin{cases}
			1 \quad &\text{if}\quad \zeta\in G_w\\
			-\epsilon \quad &\text{if}\quad \zeta\in G_w \tau\\
			0 \quad &\text{otherwise}\,,
	          \end{cases}
\end{equation*}
where $\tau\in\altGr{n}$ is the cycle  $(n, w+1, w)$ or $(1, w, w+1)$, depending on whether $w$ is equal to $1$ or not, and $G_w \tau$ denotes the right coset for the subgroup $G_w$ and the element $\tau\in\altGr{n}$, i.e. $G_w\tau$ denotes the set $\setDef{\pi\circ\tau}{\pi\in G_w}$.

We would like to guarantee that the inequality~\eqref{eqn_section-slack-covectors-extension} holds for some $\epsilon>0$, i.e.
\begin{align}
	\sum_{x\in X}\lambda_x s_{b}(x)\ge 0\quad &\text{for every}\quad b\in\mathcal{B}\label{symext:eqn_permut-slack-covectors-extension1}\\
	\sum_{x\in X}\lambda_x s_{a^i_t}(x)\ge 0\quad &\text{for every}\quad i\in\ints{\iota},t\in\ints{n}\label{symext:eqn_permut-slack-covectors-extension2}\,.
\end{align}

The left side of~\eqref{symext:eqn_permut-slack-covectors-extension1} could be rewritten as follows:
\begin{align*}
	\sum_{x\in X}\lambda_x s_{b}(x)=\sum_{\pi\in\symGr{n}}\lambda_{\Lambda(\pi)} s_{b}(\Lambda(\pi))=\sum_{\pi\in\symGr{n}}\lambda_{\Lambda(\pi)} s_{b}(\pi.\Lambda(id_n))=\\\sum_{\pi\in G_w} s_{b}(\pi.\Lambda(id_n))-\sum_{\pi\in G_w\tau}\epsilon s_{b}(\pi.\Lambda(id_n))=\sabs{G_w}(1-\epsilon)s_{b}(\Lambda(id_n))\,,
\end{align*}
which is non-negative for all $\epsilon\le 1$.
 
The left side of~\eqref{symext:eqn_permut-slack-covectors-extension2} could be rewritten as follows:
\begin{multline*}
	\sum_{x\in X}\lambda_x s_{a^i_t}(x)=\sum_{\pi\in G_w}s_{a^i_t}(\Lambda(\pi))-\sum_{\pi\in G_w\tau}\epsilon s_{a^i_t}(\Lambda(\pi))=\\
	\sum_{\pi\in G_w}s_{a^i_t}(\pi.\Lambda(id_n))-\epsilon\sum_{\pi\in G_w\tau}s_{a^i_t}(\pi.\Lambda(id_n))\,.
\end{multline*}
For $t\le w$ this expression is equal to
\begin{multline*}
	\sum_{\pi\in G_w}s_{a^i_t}(\pi.\Lambda(id_n))-\epsilon\sum_{\pi\in G_w\tau}s_{a^i_t}(\pi.\Lambda(id_n))=\\
	\sum_{v\le w}\sum_{\substack{\pi\in G_w\\\pi^{-1}(t)=v}}s_{a^i_v}(\Lambda(id_n))- \sum_{v\le { w-1}}\sum_{\substack{\pi\in G_w\tau\\\pi^{-1}(t)=v}}\epsilon s_{a^i_v}(\Lambda(id_n))- \sum_{\substack{\pi\in G_w\tau\\\pi^{-1}(t)=w+1}}\epsilon s_{a^i_v}(\Lambda(id_n))= \\        
	\frac{\sabs{G_w}}{w}(\,\sum_{v\le w}s_{a^i_v}(id_n)-\epsilon\sum_{v\le w-1}s_{a^i_v}(id_n)-\epsilon s_{a^i_{w+1}}(id_n)\,)\,.
\end{multline*}
For $t> w$ this expression is equal to
\begin{multline*}
	\sum_{\pi\in G_w}s_{a^i_t}(\pi.\Lambda(id_n))-\epsilon\sum_{\pi\in G_w\tau}s_{a^i_t}(\pi.\Lambda(id_n))=\\
	\sum_{v\ge w+1}\sum_{\substack{\pi\in G_w\\\pi^{-1}(t)=v}}s_{a^i_v}(\Lambda(id_n))- \sum_{v\ge w+2}\sum_{\substack{\pi\in G_w\tau\\\pi^{-1}(t)=v}}\epsilon s_{a^i_v}(\Lambda(id_n))- \sum_{\substack{\pi\in G_w\tau\\\pi^{-1}(t)=w}}\epsilon s_{a^i_v}(\Lambda(id_n))= \\        
	\frac{\sabs{G_w}}{n-w}(\,\sum_{v\ge w+1}s_{a^i_v}(id_n)-\epsilon\sum_{v \ge w+2}s_{a^i_v}(id_n)-\epsilon s_{a^i_{w}}(id_n)\,)\,.
\end{multline*}

Since for $w\in\ints{n-1}$ conditions~\eqref{symext:eqn_end}, \eqref{symext:eqn_start} are satisfied, we can guarantee that the above expressions are non-negative for some $\epsilon>0$.

But on the other side  for the inequality $\sum_{v\in\ints{w}} x_v\ge \frac{w(w+1)}{2}$, which is valid for the permutahedron, we obtain the inequality~\eqref{eqn_section-slack-covectors-polytope}:
\begin{multline*}
	\sum_{x\in X}\lambda_x(\sum_{v\in\ints{w}} x_v- \frac{w(w+1)}{2})=\\\sum_{\pi\in G_w}(\sum_{v\in\ints{w}} \Lambda(\pi)_v- \frac{w(w+1)}{2})-\epsilon\sum_{\pi\in G_w\tau}(\sum_{v\in\ints{w}} \Lambda(\pi)_v- \frac{w(w+1)}{2})=\\\sum_{\pi\in G_w\tau}-\epsilon<0
\end{multline*}
and finish the proof of Theorem~\ref{symext:thm_lowerbound-perm}.

\section{Cardinality Indicating Polytope}
\label{sec:symextension_card}
\begin{theorem}\label{symext:thm_lowerbound-card}
	For every~$n\ge9$ there exists no symmetric extension of the cardinality indicating polytope~$\PCard{n}$ of size less than~$\frac{n(n-1)}{2}$ with respect to the group~$G=\symGr{n}$. 
\end{theorem}

Let us introduce the operator~$\Lambda(W)$, which maps every set~$W\subseteq\ints{n}$ to the vector~$(\charVec{W}, e_{ |W|+1})$, where~$e_i\in\RR^{n+1}$ denotes the $i$-th standard basis vector. Thus, we have 
\begin{equation*}
	X=\setDef{\Lambda(W)}{W\subseteq\ints{n}}
\end{equation*} 
and for every permutation~$\pi\in \symGr{n}$ and set~$W\subseteq\ints{n}$ we have
\begin{align*}
	&(\pi.\Lambda(W))_v=\Lambda(W)_{\pi^{-1}(v)}&&\text{ for } 1 \le v \le n\\
	&(\pi.\Lambda(W))_k=\Lambda(W)_k&&\text{ for } n+1 \le k \le 2n+1\,.
\end{align*}
 For the cardinality indicating polytope the group~$\symGr{n}$ does not act transitively on the vertex set~$X$, i.e.~all vertices are divided into orbits corresponding to all possible cardinalities.

Let us assume the contrary, i.e. that there exists a symmetric subspace extension of the cardinality indicating polytope of size less than~$\frac{n(n-1)}{2}$. Applying Theorem~\ref{thm:partition_int} to the cardinality indicating polytope, we conclude that for every set~$W\subseteq\ints{n}$ the value~$s_{a^i_v}(\Lambda(W))$, $i\in\ints{\iota}$, $v\in\ints{n}$ is determined by the cardinality of the set~$W$ and  correctness of the statement~$v\in W$. In the same way, the value~$s_{b}(\Lambda(W))$, $b\in\mathcal{B}$ is determined by the cardinality of the set~$W$. Thus, we can introduce the following notation:
\begin{align*}
 c_i^0(k)&=s_{a^i_v}(\Lambda(W)) \quad&&\text{ for some } v\notin W \text{ and } |W|=k\\
 c_i^1(k)&=s_{a^i_v}(\Lambda(W)) \quad&&\text{ for some } v\in W \text{ and } |W|=k\\
 c_b(k)&=s_{b}(\Lambda(W)) \quad&&\text{ for some } |W|=k\,,
\end{align*}
which are non-negative values.

\begin{lemma}\label{symext:lem_exist-cardinality}
 There exists a cardinality~$k^*\in\ints{n-1}$  such that
\begin{equation}\label{symext:eqn_card}
	\text{if}  \quad c^0_i(k^*)>0 \text{  or  }c^1_i(k^*)> 0 \quad  \text{then} \sum_{0 \le k < k^*}c^0_i(k)+\sum_{k^*<k \le n}c^1 _i(k)>0
\end{equation}
holds for all $i\in\ints{\iota}$.
\end{lemma}
\begin{proof}
 Since each set~$\mathcal{A}_i$ consists of~$n$ components, we can conclude that $\iota$ is smaller than~$\frac{n-1}{2}$     (recall~$d<\frac{n(n-1)}{2}$). 

For each set~$i\in\ints{\iota}$ there are not more than two cardinalities in $\ints{n-1}$, which do not satisfy~\eqref{symext:eqn_card}.  To prove this assign to~$i$  the minimum cardinality $k^i_{\min}$ and the maximum cardinality $k^i_{\max}$ for which the statement~\eqref{symext:eqn_card} is violated. From \eqref{symext:eqn_card} we can conclude that for all~$k$, $k^i_{\min}< k<k^i_{\max}$ the values $c^0_i(k)$ and $c^1_i(k)$ are equal to $0$. Thus, for all~$k^i_{\min}< k<k^i_{\max}$ the statement~\eqref{symext:eqn_card} holds.

This shows, that there exists at least one cardinality from $1$ till $n-1$ which satisfies \eqref{symext:eqn_card} for all $i\in\ints{\iota}$.
\end{proof}

To construct a contradiction to the assumption that there is an extension with the above properties we use~\eqref{eqn_section-slack-covectors-extension} and \eqref{eqn_section-slack-covectors-polytope}. Here, we choose $\lambda_x$, $x\in X$, where $x=\Lambda(W)$, $W\subseteq \ints{n}$ as follows:
\begin{equation*}
	\lambda_x=\begin{cases}
			1 \quad&\text{if}\quad W=\ints{t}, 0\le t\le n, t\not = k^*\\
			1+\epsilon \quad&\text{if}\quad W=\ints{k^*}\\
			-\epsilon \quad&\text{if}\quad W=\ints{k^*-1}\cap\{k^*+1\}\\
			0 \quad&\text{otherwise}\,.
	          \end{cases}
\end{equation*}

We would like to guarantee that the inequality~\eqref{eqn_section-slack-covectors-extension} holds for some $\epsilon>0$, i.e.
\begin{align}
	\sum_{x\in X}\lambda_x s_{b}(x)\ge 0\quad &\text{for every}\quad b\in\mathcal{B}\label{symext:eqn_card-slack-covectors-extension1}\\
	\sum_{x\in X}\lambda_x s_{a^i_t}(x)\ge 0\quad &\text{for every}\quad i\in\ints{\iota}, t\in\ints{n}\label{symext:eqn_card-slack-covectors-extension2}\,.
\end{align}

The left side of~\eqref{symext:eqn_card-slack-covectors-extension1} could be rewritten as follows:
\begin{align*}
	\sum_{x\in X}\lambda_x s_{b}(x)=\sum_{0\le k\le n}c_b(k) +\epsilon c_b(k^*)- \epsilon c_b(k^*)=\sum_{0\le k\le n}c_b(k)\,,
\end{align*}
which is non-negative for all $\epsilon$.
 
The left side of~\eqref{symext:eqn_card-slack-covectors-extension2} for  $t\notin \{k^*, k^*+1\}$ is equal to :
\begin{align*}
	\sum_{x\in X}\lambda_x s_{a_t^i}(x)=\sum_{0\le k\le t-1}c^0_i(k) + \sum_{t\le k\le n}c^1_i(k)
\end{align*}
and for  $t= k^*$ is equal to :
\begin{multline*}
	\sum_{x\in X}\lambda_x s_{a_t^i}(x)=\sum_{0\le k\le k^*-1}c^0_i(k) + \sum_{k^*\le k\le n}c^1_i(k) - \epsilon c^0_i(k^*)+ \epsilon c^1_i(k^*)=\\\sum_{0\le k< k^*}c^0_i(k) + \sum_{k^*< k\le n}c^1_i(k) - \epsilon c^0_i(k^*)+ (1+\epsilon) c^1_i(k^*)
\end{multline*}
and for  $t= k^*+1$ is equal to :
\begin{multline*}
	\sum_{x\in X}\lambda_x s_{a_t^i}(x)=\sum_{0\le k\le k^*}c^0_i(k) + \sum_{k^*+1\le k\le n}c^1_i(k) - \epsilon c^1_i(k^*)+ \epsilon c^0_i(k^*)=\\\sum_{0\le k< k^*}c^0_i(k) + \sum_{k^*< k\le n}c^1_i(k) - \epsilon c^1_i(k^*)+ (1+\epsilon) c^0_i(k^*)\,.
\end{multline*}
Due to~\eqref{symext:eqn_card} there exists $\epsilon>0$ such that all above expressions are non-negative.

Let us use the inequality
\begin{equation*}
	 \sum_{1\le v\le k^*} x_v -\sum_{1\le k\le k^*}{k z_k}-\sum_{k^*< k\le n}{k^* z_k}\le 0\,,
\end{equation*}
which is valid for $\PCard{n}$, as the inequality in the condition~\eqref{eqn_section-slack-covectors-polytope}. For all vertices $x\in X$ except $\Lambda(\ints{k^*-1}\cup\{k^*+1\})$, the coefficient $\lambda_x$  or the value $-\sum_{1\le v\le k^*} x_v +\sum_{1\le k\le k^*}{k z_k}+\sum_{k^*< k\le n}{k^* z_k}$ is equal to $0$, and thus
\begin{equation*}
 	\sum_{x\in X} \lambda_x (-\sum_{1\le v\le k^*} x_v +\sum_{1\le k\le k^*}{k z_k}+\sum_{k^*< k\le n}{k^* z_k})=\lambda_{\Lambda{\ints{k^*-1}\cap\{k^*+1\}}}=-\epsilon<0\,,
\end{equation*}
which finishes the proof.

\bibliographystyle{plain}
\bibliography{permutahedra-main}

\end{document}